\documentclass[leqno,11pt]{article}

\usepackage{amsmath}
\usepackage{epsfig}
\usepackage[utf8]{inputenc}
\usepackage{amssymb}
\usepackage[all]{xy}
\usepackage{color}
\newcommand{\qed}{\penalty 500\hfill$\square$\par\medskip}

\newcommand{\dd}{\mathrm{d}}

\newtheorem{theorem}{Theorem}
\newtheorem{proposition}{Proposition}
\newtheorem{definition}{Definition}
\newtheorem{lemma}{Lemma}
\newtheorem{corollary}{Corollary}

\begin{document}

\title{\Large{\bf A problem involving competing and intrinsic operators}}

\author{\sc
Aldo H. S. Medeiros
\thanks{Departamento de Matem\'{a}tica,
Universidade Federal de Viçosa, 36570-900 - Vi\c{c}osa - MG,
Brazil. E-mail: aldo.medeiros@ufv.br} \,\,\, Dumitru Motreanu
\thanks{Department of Mathematics, University of Perpignan, 66000 Perpignan, France.
E-mail: mtreanu@univ-perp.fr} }
\date{}

\maketitle

\setcounter{equation}{0}
%%%%%%%%%%%%%%%%%%%%%%%%%%%%%%%%%%%%%%%%%%%%%%%%%%%%%%%%%%%%%%%%%
\begin{abstract}

\end{abstract}

The main result establishes the existence of a solution in a
generalized sense for a nonlinear Dirichlet problem driven by a
competing operator and exhibiting a convection term composed with an
intrinsic operator. A finite dimensional approximation is
constructed. Applications regarding nonhomogeneous Dirichlet
problems and equations with convolution are given by choosing an
adequate intrinsic operator.

\

\noindent {\bf $2010$ Mathematics Subject Classification:} 35J92
47F10 35B45

\


\

\date{}

\maketitle

\section{Introduction and statement of main result}\label{S1}

Let $T:W_0^{1, p}(\Omega) \rightarrow W^{1, p}(\Omega)$ be a given
continuous map that we call intrinsic operator. For $1<q<p<+\infty$
and $p<N$, we state the Dirichlet problem
\begin{equation}\label{P1}
\begin{cases}-\Delta_p u + \Delta_q u=f(x, T(u), \nabla T(u)) & \text { in } \Omega \\
u=0 & \text { on } \partial \Omega\end{cases}
\end{equation}
on a bounded domain $\Omega \subset \mathbb{R}^N$ with a Lipschitz
boundary $\partial\Omega$. In the case where the operator $T$ is the
identity of $W_0^{1, p}(\Omega)$ we recover the problem stated in
\cite{M}.

In order to simplify the notation we make the convention that
corresponding to any $r\in (1,+\infty)$ we denote $r'=r/(r-1)$ (the
H\"older conjugate of $r$).  In \eqref{P1} we have the negative
$p$-Laplacian $-\Delta_p: W_0^{1, p}(\Omega) \rightarrow W^{-1,
p^{\prime}}(\Omega)$ given by
\begin{eqnarray*}
\left\langle-\Delta_p u, v\right\rangle=\int_{\Omega}|\nabla
u(x)|^{p-2} \nabla u(x) \cdot \nabla v(x) d x, \quad \forall u,v \in
W_0^{1, p}(\Omega),
\end{eqnarray*}
and the positive $q$-Laplacian $\Delta_q: W_0^{1, q}(\Omega)
\rightarrow W^{-1, q^{\prime}}(\Omega)$ defined by
\begin{eqnarray*}
\left\langle \Delta_q u, v\right\rangle=-\int_{\Omega}|\nabla
u(x)|^{q-2} \nabla u(x) \cdot \nabla v(x) d x, \quad \forall u,v \in
W_0^{1, q}(\Omega) .
\end{eqnarray*}
Since $p>q$, there is the continuous embedding $W_0^{1, p}(\Omega)
\hookrightarrow W_0^{1, q}(\Omega)$, so the sum of operators
$-\Delta_p + \Delta_q:W_0^{1, p}(\Omega) \rightarrow W^{-1,
p^{\prime}}(\Omega)$ is well defined . The case $q>p$ in problem
\eqref{P1} is obtained multiplying the equation by $-1$. The case
$N\leq p$ is simpler and can be handled along the same lines.

In view of the Sobolev and Rellich-Kondrachov embedding theorems, the
Sobolev space $W_0^{1, p}(\Omega)$ is compactly embedded into
$L^r(\Omega)$ for every $1 \leq r<p^*$ and continuously embedded for
$r=p^*$, where $p^*$ stands for the Sobolev critical exponent
$p^*=\frac{N p}{N-p}$ if $N>p$. Hence, for every $1 \leq r \leq p^*$
there exists a positive constant $S_r$ such that
\begin{equation}\label{eq1}
\|u\|_r \leq S_r\|u\|, \quad \forall u \in W_0^{1, p}(\Omega),
\end{equation}
where $\|u\|_r$ denotes the norm on $L^r(\Omega)$ and
\[
\Vert u \Vert := \Vert \nabla u \Vert_{p} =
\displaystyle\left(\int_{\Omega} \vert \nabla u \vert^{p} \dd
x\right)^{\frac{1}{p}}
\]
is the norm on $W_0^{1, p}(\Omega)$. The best constant in
\eqref{eq1} for $r=p$ is $S_p=\lambda_{1, p}^{-1 / p}$, where
$\lambda_{1, p}$ is the first eigenvalue of $-\Delta_p$, that is,
\begin{equation*}
\lambda_{1,p}=\inf_{u \in W_0^{1,p}(\Omega),u \neq 0} \frac{\|\nabla
u\|_p^p}{\|u\|_p^p}.
\end{equation*}

The right-hand side of the equation in problem \eqref{P1} is
composed of a Carath\'eodory function $f:\Omega \times \mathbb{R}
\times$ $\mathbb{R}^N \rightarrow \mathbb{R}$ (i.e., $f(x,s,\xi)$ is
measurable in $x\in\Omega$ and continuous in $(s,\xi)\in \mathbb{R}
\times$ $\mathbb{R}^N$) and the continuous map
$T:W_0^{1,p}(\Omega)\rightarrow W^{1, p}(\Omega)$. We assume that
the following conditions are satisfied:
\begin{enumerate}
\item[$(H1)$] There exist constants $a_1 \geq 0$, $a_2
\geq 0$, $\alpha \in\left(0, p^*-1\right)$, $\beta \in\left(0,
\frac{p}{\left(p^*\right)^{\prime}}\right)$, and a function
$\sigma\in L^{r^{\prime}}(\Omega)$ with $r\in\left[1, p^*\right)$
such that
$$
|f(x, s, \xi)| \leq \sigma_1(x)+a_1|s|^\alpha+a_2|\xi|^\beta
$$
for almost all $x \in \Omega$, all $s \in \mathbb{R}$ and $\xi \in
\mathbb{R}^N$.
\item[$(H2)$] There exist constants $K_1, K_2, K_3 \geq 0$ with
$$
a_1 K_1 S_{\frac{p^*}{p^*-\alpha}}+a_2 K_2 S_{\frac{p}{p-\beta}}< 1
$$
such that
$$
\|T(u)\|_{p^*}^\alpha \leq K_1\|\nabla u\|_p^{p-1}+K_3 \quad \text {
and } \quad\|\nabla T(u)\|_p^\beta \leq K_2\|\nabla u\|_p^{p-1}+K_3
$$
for all $u \in W_0^{1, p}(\Omega)$.
\end{enumerate}

The solutions to problem \eqref{P1} are understood in the following
sense consistent with \cite{M}.

\begin{definition}\label{D1}
A function $u \in W_0^{1, p}(\Omega)$ is said to be a generalized
solution to problem \eqref{P1} if there exists a sequence
$\left\{u_n\right\}_{n \in \mathbb{N}}$ in $W_0^{1, p}(\Omega)$ such
that
\begin{enumerate}
\item[$(a)$] $u_n \rightharpoonup u$ in $W_0^{1, p}(\Omega)$ as $n \rightarrow \infty$;
\item[$(b)$] $-\Delta_p u_n+\Delta_q u_n-f\left(\cdot, T(u_n)(\cdot),
\nabla T(u_n)(\cdot)\right) \rightharpoonup 0$ in $W^{-1,
p^{\prime}}(\Omega)$ as $n \rightarrow \infty$;
\item[$(c)$] $\displaystyle\lim _{n \rightarrow \infty}\left[\left\langle
-\Delta_p u_n + \Delta_q u_n, u_n-u\right\rangle-\int_{\Omega}
f\left(x, T(u_n), \nabla T(u_n)\right)\left(u_n-u\right) \dd
x\right]=0$.
\end{enumerate}
If in place of (c) it holds
\[
\displaystyle\lim _{n \rightarrow \infty}\left\langle-\Delta_p u_n +
\Delta_q u_n, u_n-u\right\rangle=0,
\]
we say that $u$ is a strong generalized solution of problem
\eqref{P1}.
\end{definition}

Hypothesis $(H1)$ ensures that Definition \ref{D1} is correctly
formulated, in particular the existence of the integral in (c) is
assured. This can be seen from the lemma below.

\begin{lemma}\label{L1}
Under assumption $(H1)$ one has the estimate
\[
\left\vert \displaystyle\int_{\Omega} f(x,T(u), \nabla T(u)) v \dd x
\right\vert \leq \Vert \sigma_1\Vert_{r^{\prime}}\Vert v
\Vert_r+a_1\Vert T(u)\Vert_{p^*}^\alpha \Vert v
\Vert_{\frac{p^*}{p^*-\alpha}}+a_2\Vert \nabla T(u)\Vert^\beta \Vert
v \Vert_{\frac{p}{p-\beta}} .
\]
for all $u,v \in W_{0}^{1,p}(\Omega)$.
\end{lemma}

\begin{proof}
Assumption (H1) and H\"older's inequality entail
\begin{align*}
\left|\int_{\Omega} f(x, T(u), \nabla T(u)) v \dd x\right| & \leq
\int_{\Omega}|\sigma \| v| \dd x
+ a_1 \int_{\Omega}|T(u)|^{\alpha}|v| \dd x
+ a_2 \int_{\Omega}|\nabla T(u)|^{\beta}|v| \dd x \\
&
\leq\|\sigma\|_{r'}\|v\|_{r}
+a_1\|T(u)\|_{p^*}^{\alpha}\|v\|_{\frac{p^*}{p^*-\alpha}}+a_2\|\nabla
T(u)\|_{p}^{\beta}\|v\|_{\frac{p}{p-\beta}}, \quad \forall u, v \in
W_0^{1, p}(\Omega),
\end{align*}
which is the desired conclusion.
\end{proof}
\qed

Our main result provides the existence of a strong generalized
solution to problem \eqref{P1}.

\begin{theorem}\label{T1}
Assume that conditions $(H1)-(H2)$ hold. Then there exists a strong
generalized solution $u \in W_0^{1,p}(\Omega)$ to problem
\eqref{P1}.
\end{theorem}

The arguments developed in Sections \ref{S2} and \ref{S3} to prove
Theorem \ref{T1} permit to establish a slightly more general version
of Theorem \ref{T1}. Since the reasoning is basically the same, we
omit the proof.

\begin{corollary}\label{C1}
Under hypotheses $(H1)-(H2)$, the Dirichlet problem
\begin{equation}\label{P2}
\begin{cases}-{\rm div}|\nabla(w+u_0)|^{p-2}\nabla (w+u_0)
+ {\rm div}|\nabla(w+u_0)|^{q-2}\nabla (w+u_0)=f(x, T(w), \nabla T(w)) & \text { in } \Omega \\
w=0 & \text { on } \partial \Omega\end{cases},
\end{equation}
for some $u_0\in  W^{1,p}(\Omega)$, admits a strong generalized
solution $w \in W_0^{1,p}(\Omega)$ in the following sense: there
exists a sequence $\left\{w_n\right\}_{n \in \mathbb{N}}$ in
$W_0^{1, p}(\Omega)$ such that
\begin{enumerate}
\item[$(\tilde{a})$] $w_n \rightharpoonup w$ in $W_0^{1, p}(\Omega)$ as $n \rightarrow \infty$;
\item[$(\tilde{b})$] $-{\rm div}|\nabla(w_n+u_0)|^{p-2}\nabla(w_n+u_0)
+ {\rm div}|\nabla(w_n+u_0)|^{q-2}\nabla (w_n+u_0)-f\left(\cdot,
T(w_n)(\cdot), \nabla T(w_n)(\cdot)\right) \rightharpoonup 0$ in
$W^{-1, p^{\prime}}(\Omega)$ as $n \rightarrow \infty$;
\item[$(\tilde{c})$]
$\langle -{\rm div}(|\nabla(w_n+u_0)|^{p-2}\nabla (w_n+u_0) +
|\nabla(w_n+u_0)|^{q-2}\nabla (w_n+u_0)), w_n-w\rangle \rightarrow
0$ as $n \rightarrow \infty$.
\end{enumerate}
If $u_0=0$, Theorem \ref{T1} is recovered.
\end{corollary}

The proof of Theorem \ref{T1} is given in Section \ref{S3}.
Applications of Theorem \ref{T1} are presented in Section \ref{S4}.

The present paper aims to address two major difficulties that until
now been regarded regarded separately. The first one concerns the
competing operator $-\Delta_p + \Delta_q$ which destroys the
ellipticity of problem \eqref{P1} and has a totally different
behavior than the negative $(p,q)$-Laplacian $-\Delta_p - \Delta_q$.
The second one refers to the convection term
$f(\cdot,u(\cdot),\nabla u(\cdot))$ (i.e., it depends on the
solution $u$ and its gradient $\nabla u$) composed with an
additional map $T: W_0^{1, p}(\Omega) \rightarrow W^{1, p}(\Omega)$
called intrinsic operator, which can cover a multitude of local and
nonlocally relevant phenomena. It is for the first time when
problems driven by the competing operator $-\Delta_p + \Delta_q$ and
incorporating an intrinsic operator $T: W_0^{1, p}(\Omega)
\rightarrow W^{1, p}(\Omega)$ are studied.

We briefly outline related literature. The study of problems
involving the competing operator $-\Delta_p + \Delta_q$ started with
\cite{LLMZ} in a variational setting. Non-variational problems
exhibiting competing operators and convection terms were
investigated in \cite{M}. The problems with intrinsic operator have
been introduced in \cite{MM} where the driving operator is the
ordinary $(p,q)$-Laplacian. Equations with convolution in the
convection term were first considered in \cite{MotreanuMotreanu2}
and subsequently studied in \cite{MotreanuMotreanu} and \cite{MVV}.
In this paper, all the described elements are involved in the same
problem statement.

Our main result, stated as Theorem \ref{T1}, provides the existence
of solutions to problem \eqref{P1} in the sense of Definition
\ref{D1}. The proof relies on an approximation approach based on
resolving finite-dimensional equations via a consequence of
Brouwer's fixed point theorem. An essential feature of our arguments
is to appropriately match the growth of the intrinsic operator and
the growth of the convection in the a priori estimates. The abstract
result is applied to Dirichlet problems with nonhomogeneous boundary
condition, and to equations with convolution incorporated in the
convection term.

The rest of the paper is organized as follows: Section \ref{S2} is
devoted to finding finite dimensional approximations. Section
\ref{S2} contains the proof of our main result. Section \ref{S3}
sets forth two applications of the main result.

\section{Approximate solutions on finite dimensional
subspaces}\label{S2}

A consequence of Brouwer's fixed point theorem will be an essential
tool in the proof of Theorem \ref{T1}. The proof of this consequence
is usually done on $\mathbb{R}^n$ endowed with the standard
Euclidean norm (see, e.g., \cite[p. 37]{Ralph}). For the sake of
precision, we prove it on an arbitrary finite dimensional space $V$
as will be needed in the sequel without claiming that it directly
follows from the identification of $V$ with $\mathbb{R}^n$.

\begin{lemma}\label{L2}
Let $V$ be a finite dimensional space endowed with the norm
$\|\cdot\|_V$ and let $F: V \rightarrow V^{\star}$ be a continuous
map. Assume that there is a constant $R>0$ such that
\begin{eqnarray}\label{1}
\langle F(v), v\rangle \geq 0 \text { for all } v \in
V \text { with }\|v\|_V=R.
\end{eqnarray}
Then there exists $u \in V$ with $\|u\|_V \leq R$ satisfying
$F(u)=0$.
\end{lemma}

\begin{proof}
Arguing by contradiction, suppose that $F(v)\not=0$  for all
$\|v\|_V \leq R$. Fix a basis $\{v_1,\dots,v_n\}$ of $V$, where $n$
is the dimension of $V$, and define the map
$G:\mathbb{R}^n\rightarrow \mathbb{R}^n$ by
$$
G(x)=(\langle F(x),v_1\rangle,\dots,\langle F(x),v_n\rangle) \quad
\text { for all } x=(x_1,\dots,x_n)\in \mathbb{R}^n.
$$
On the standard product $G(x)\cdot x$ in $\mathbb{R}^n$ we infer by
\eqref{1} that
$$
G(x)\cdot x=\left\langle
F\left(\sum_{i=1}^nx_iv_i\right),\sum_{i=1}^nx_iv_i\right\rangle\geq
0
$$
for all $x=(x_1,\dots,x_n)\in \mathbb{R}^n$ with
$|x|_1:=\|\sum_{i=1}^nx_iv_i\|_V=R$. Notice that $|\cdot|_1$ is a
norm on $\mathbb{R}^n$ and consider the corresponding closed ball
$B_R=\left\{x\in \mathbb{R}^n:|x|_1\leq R\right\}$. In view of our
assumption in arguing by contradiction, the map $g:B_R\rightarrow
B_R$ given by
$$
g(x)=-\frac{R}{|G(x)|_1}G(x) \quad \text { for all } x\in B_R
$$
is well defined and continuous. Consequently, Brouwer's fixed point
theorem ensures the existence of some $x_0\in B_R$ satisfying
$g(x_0)=x_0$. Since $|x_0|_1=|g(x_0)|_1=R$, we have $x_0\not=0$. On
the other hand, denoting by $|\cdot|_2$ the Euclidean norm on
$\mathbb{R}^n$, there exists a constant $c>0$ such that $|x|_1\leq
c|x|_2$ whenever $x\in \mathbb{R}^n$. This yields the contradiction
$$
0<R^2=|x_0|_1^2\leq c^2|x_0|_2^2=c^2x_0\cdot x_0=c^2x_0\cdot
g(x_0)=-\frac{Rc^2}{|G(x_0)|_1}G(x_0)\cdot x_0\leq 0,
$$
which proves the result.
\end{proof}
\qed

The construction of a sequence as required in Definition \ref{D1}
relies on the usage of a Galerkin basis of the Sobolev space
$W_0^{1, p}(\Omega)$. We recall that a Galerkin basis of $W_0^{1,
p}(\Omega)$ means a sequence $\left\{X_n\right\}_{n \in \mathbb{N}}$
of vector subspaces of $W_0^{1, p}(\Omega)$ satisfying
\begin{enumerate}
\item[$(i)$] $\operatorname{dim}\left(X_n\right)<\infty, \quad \forall n \in \mathbb{N}$;
\item[$(ii)$] $X_n \subset X_{n+1}, \quad \forall n \in \mathbb{N}$;
\item[$(iii)$] $\displaystyle \overline{\bigcup_{n=1}^{\infty} X_n}=W_0^{1, p}(\Omega)$.
\end{enumerate}
The existence of a Galerkin basis is guaranteed because the Banach
space $W_0^{1, p}(\Omega)$ with $1<p<+\infty$ is separable. Fix a
Galerkin basis $\left\{X_n\right\}_{n \in \mathbb{N}}$ of $W_0^{1,
p}(\Omega)$. Lemma \ref{L2} will be applied on each finite
dimensional subspace $X_n$.

\begin{proposition}\label{Pr1}
Assume that conditions $(H1)-(H2)$ are fulfilled. Then for each $n
\in \mathbb{N}$ there exists $u_n \in X_n$ such that
\begin{equation}\label{eq3}
\left\langle-\Delta_p u_n+\Delta_q u_n, v\right\rangle=\int_{\Omega}
f\left(x, T(u_n), \nabla T(u_n)\right) v \mathrm{d} x \quad \text {
for all } v \in X_n .
\end{equation}
\end{proposition}

\begin{proof}
For each $n \geq 1$ we introduce the map $A_n: X_n \rightarrow
X_n^*$ by
$$
\left\langle A_n(u), v\right\rangle=\left\langle-\Delta_p u+\Delta_q
u, v\right\rangle-\int_{\Omega} f(x, T(u), \nabla T(u)) v \dd x
\text { for all } u, v \in X_n .
$$
On the basis of Lemma \ref{L1}, assumption $(H2)$, \eqref{eq1}, and
H\"older's inequality, we find that
$$
\begin{aligned}
\left\langle A_n(v), v\right\rangle &
=\int_{\Omega}\bigg(|\nabla v|^p-|\nabla v|^q-f(x, T(v), \nabla T(v)) v\bigg) \dd x \\
& \geq \Vert \nabla v \Vert_{p}^{p} - \Vert \nabla v \Vert_{q}^{q} -
\|\sigma\|_{r'}\|v\|_{r} -
a_1\|T(v)\|_{p^*}^{\alpha}\|v\|_{\frac{p^*}{p^*-\alpha}}
- a_2\|\nabla T(v)\|_{p}^{\beta}\|v\|_{\frac{p}{p-\beta}}\\
& \geq \Vert \nabla v \Vert_{p}^{p} - \|\nabla v\|_{q}^q
- S_r\|\sigma\|_{r'}\|\nabla v\|_{p} - a_1(K_1\|v\|^{p-1}+K_3)\|v\|_{\frac{p^*}{p^*-\alpha}}
- a_2(K_2\|v\|^{p-1}+K_3)\|v\|_{\frac{p}{p-\beta}}\\
& \geq\left(1- a_1 K_1 S_{\frac{p^*}{p^*-\alpha}}-a_2 K_2
S_{\frac{p}{p-\beta}}\right)\|\nabla
v\|_{p}^p-|\Omega|^{\frac{p-q}{p}}\|\nabla v\|_{p}^q-C_1\|\nabla
v\|_{p}, \quad \forall v \in X_n,
\end{aligned}
$$
where $C_1>0$ is a constant and $|\Omega|$ denotes the Lebesgue
measure of $\Omega$. Using that $p>q>1$ and $1- a_1 K_1
S_{\frac{p^*}{p^*-\alpha}}-a_2 K_2 S_{\frac{p}{p-\beta}}>0$ as known
from $(H2)$, we have
\[
\langle A_n(v),v \rangle \geq 0, \ \ \text{for all} \ \ v \in X_n \
\ \text{with} \ \ \Vert v \Vert = R
\]
provided $R>0$ is sufficiently large. Lemma \ref{L2} can thus be
applied with $V=X_n$ and $F = A_n$. Consequently, there exists $u_n
\in X_n$ solving the equation $A_n (u_n) = 0$, so \eqref{eq3} holds
true. The proof is complete.
\end{proof}
\qed

The element $u_n \in X_n$ appearing in Proposition \ref{Pr1} can be
interpreted as a finite dimensional approximate solution to problem
\eqref{P1}. The next corollary establishes an important property of
the sequence $\left\{u_n\right\}_{n \in \mathbb{N}}$.
\begin{corollary}\label{C2}
Assume that conditions $(H1)-(H2)$ are fulfilled. Then the sequence
$\left\{u_n\right\}_{n \in \mathbb{N}}$ constructed in Proposition
\ref{Pr1} is bounded in $W_0^{1, p}(\Omega)$.
\end{corollary}
\begin{proof}
Let us insert $v=u_n$ in \eqref{eq3}. Through H\"older's inequality,
Lemma \ref{L1}, \eqref{eq1}, and assumption $(H2)$, we derive as
before
\begin{align*}
\left\|\nabla u_n\right\|_{p}^p&=\left\|\nabla u_n\right\|_{q}^q
+\int_{\Omega} f\left(x, T(u_n), \nabla T(u_n)\right) u_n \dd x \\
&\leq|\Omega|^{\frac{p-q}{p}}\left\|\nabla u_n\right\|_{p}^q +
\|\sigma\|_{r'}\|u_n\|_{r}+a_1\|T(u_n)\|_{p^*}^{\alpha}\|u_n\|_{\frac{p^*}{p^*-\alpha}}
+a_2\|\nabla T(u_n)\|_{p}^{\beta}\|u_n\|_{\frac{p}{p-\beta}}\\
&\leq|\Omega|^{\frac{p-q}{p}}\left\|\nabla u_n\right\|_{p}^q +
S_r\|\sigma\|_{r'}\|\nabla u_n\|_{p} + \left(a_1 K_1
S_{\frac{p^*}{p^*-\alpha}}+a_2 K_2
S_{\frac{p}{p-\beta}}\right)\|\nabla u_n\|_{p}^p + C_1\|\nabla
u_n\|_{p},
\end{align*}
with a constant $C_1>0$. Since $a_1 K_1
S_{\frac{p^*}{p^*-\alpha}}+a_2 K_2 S_{\frac{p}{p-\beta}}<1$, it
results the desired conclusion.
\end{proof}
\qed

\section{Proof of Theorem \ref{T1}}\label{S3}

Corollary \ref{C2} asserts that the sequence $\left\{u_n\right\}_{n
\in \mathbb{N}}$ constructed in Proposition \ref{Pr1} is bounded in
$W_0^{1, p}(\Omega)$. Since the space $W_0^{1, p}(\Omega)$ is
reflexive, along a relabeled subsequence we have $u_n
\rightharpoonup u$ in $W_0^{1, p}(\Omega)$ with some $u \in W_0^{1,
p}(\Omega)$. Condition $(a)$ in Definition \ref{D1} is thus
verified.

Due to hypothesis $(H1)$, the Nemytskii operator $\mathcal{N}_f:
W^{1, p}(\Omega) \rightarrow$
$L^{\left(p^*\right)^{\prime}}(\Omega)$ asssociated with the
function $f$, that is
$$
\mathcal{N}_f(u)=f(x, u, \nabla u), \quad \forall u\in W^{1,
p}(\Omega),
$$
is well defined, bounded (in the sense that it maps bounded sets
into bounded sets) and continuous as shown in Lemma \ref{L1}. Hence,
by hypothesis $(H2)$ and the continuity of the embedding
$L^{\left(p^*\right)^{\prime}}(\Omega) \hookrightarrow W^{-1,
p^{\prime}}(\Omega)$, the operator $\mathcal{N}_f \circ T: W_0^{1,
p}(\Omega) \rightarrow W^{-1, p^{\prime}}(\Omega)$ is continuous and
bounded. Taking into account that $1<q<p$, one has the continuous
embedding $W_0^{1, p}(\Omega) \hookrightarrow W_0^{1, q}(\Omega)$,
so the operator $-\Delta_p + \Delta_q: W_0^{1, p}(\Omega)
\rightarrow W^{-1, p^{\prime}}(\Omega)$ is continuous and bounded.
Consequently, the operator $ -\Delta_p + \Delta_q - \mathcal{N}_f
\circ T$ is bounded, which enables us to conclude that the sequence
$\left\{-\Delta_p u_n+\Delta_q u_n-\mathcal{N}_f \circ
T\left(u_n\right)\right\}_{n \geq 1}$ is bounded in $W^{-1,
p^{\prime}}(\Omega)$. The reflexivity of $W^{-1,
p^{\prime}}(\Omega)$ allows us to pass to a relabeled subsequence of
$\left\{u_n\right\}_{n \in \mathbb{N}}$ for which
$$
-\Delta_p u_n+\Delta_q u_n-\mathcal{N}_f \circ T\left(u_n\right)
\rightharpoonup \eta \text { in } W^{-1, p^{\prime}}(\Omega)
$$
with some $\eta \in W^{-1, p^{\prime}}(\Omega)$.

Let $v \in \displaystyle\bigcup_{n\in\mathbb{N}} X_n$ and choose an
integer $m \geq 1$ such that $v \in X_m$. Proposition \ref{Pr1}
entails that equality \eqref{eq3} holds true for all $n \geq m$
because $X_k \subset X_{k+1}$ for all $k$. Letting $n \rightarrow
\infty$ in \eqref{eq3} we obtain $\langle\eta, v\rangle=0$. Then the
density of $\displaystyle\bigcup_{n \in \mathbb{N}} X_n$ in
$W_0^{1,p}(\Omega)$ implies $\eta=0$. Consequently, property $(b)$
in Definition \ref{D1} is proven.

Inserting $v=u_n$ in \eqref{eq3} gives
\begin{equation}\label{eq5}
\left\|\nabla u_n\right\|_{p}^p=\left\|\nabla u_n\right\|_{q}^q
+\int_{\Omega} f\left(x, T(u_n), \nabla T(u_n)\right) u_n \dd x,
\quad \forall n \geq 1,
\end{equation}
while by property $(b)$ in Definition \ref{D1} we get
\begin{equation}\label{eq6}
\left\langle-\Delta_p u_n+\Delta_q u_n-\mathcal{N}_f \circ
T\left(u_n\right), u\right\rangle \rightarrow 0 \text { as } n
\rightarrow \infty.
\end{equation}
From \eqref{eq5} and \eqref{eq6} it turns out
\begin{eqnarray}\label{4}
\lim _{n \rightarrow \infty}\left\langle-\Delta_p u_n +\Delta_q
u_n-\mathcal{N}_f \circ T\left(u_n\right), u_n-u\right\rangle=0.
\end{eqnarray}

Due to assumption $(H_2)$ and the boundedness of
$\left\{u_n\right\}_{n \in \mathbb{N}}$ there exists a constant
$c_1>0$ such that $\Vert T(u_n) \Vert_{p^*}\leq c_1$ and $\Vert
\nabla T(u_n) \Vert_p \leq c_1$ for all $n$. Then Lemma \ref{L1}
leads to the estimate
\begin{align*}
\left|\langle\mathcal{N}_f \circ T(u_n),u_n-u \rangle\right|
&= \left|\int_{\Omega} f\left(x, T\left(u_n\right),
\nabla T\left(u_n\right)\right)\left(u_n-u\right) d x\right|\\
&\leq \Vert \sigma_1\Vert_{r^{\prime}}\Vert u_n - u \Vert_r
+a_1\Vert T(u_n)\Vert_{p^*}^\alpha \Vert u_n-u \Vert_{\frac{p^*}{p^*-\alpha}}
+a_2\Vert \nabla T(u_n)\Vert_{p}^\beta \Vert u_n-u \Vert_{\frac{p}{p-\beta}} \\
&\leq
\left\|\sigma_1\right\|_{r^{\prime}}\left\|u_n-u\right\|_r+a_1
c_1^\alpha\left\|u_n-u\right\|_{\frac{p^*}{p^*-\alpha}}+a_2
c_1^\beta\left\|u_n-u\right\|_{\frac{p}{p-\beta}}
\end{align*}
for every $n$. In view of $(H1)$, the space $W_0^{1, p}(\Omega)$ is
compactly embedded in $L^r(\Omega)$, $L^{p^*
/\left(p^*-\alpha\right)}(\Omega)$, and $L^{p /(p-\beta)}(\Omega)$.
Letting $n \rightarrow+\infty$ we arrive at
\[
\lim _{n \rightarrow \infty}\langle\mathcal{N}_f \circ T(u_n),u_n-u \rangle
= \displaystyle\lim_{n \to \infty}\int_{\Omega} f\left(x, T(u_n), \nabla T(u_n)\right) u_n \dd x = 0.
\]
Then \eqref{4} yields
\[
\lim _{n \rightarrow \infty}\left\langle-\Delta_p u_n+\Delta_q u_n,
u_n-u\right\rangle=0.
\]
Therefore according to Definition \ref{D1}, $u$ is a strong
generalized solution of the problem \eqref{P1}, which completes the
proof of Theorem \ref{T1}. \qed

\section{Applications}\label{S4}

Our first application refers to the nonhomogenous Dirichlet problem
involving competing operators and convection
\begin{equation}\label{P3}
\begin{cases}-\Delta_p u + \Delta_q u=f(x, u, \nabla u) & \text { in } \Omega \\
u=u_0 & \text { on } \partial \Omega,\end{cases}
\end{equation}
with $u_0\in W^{1,p}(\Omega)$. The data $\Omega$, $\Delta_p$,
$\Delta_q$, and $f(x,s,\xi)$ have the same significance as in the
statement of problem \eqref{P1}. The equality $u=u_0$ on the
boundary $\partial \Omega$ is in the sense of traces meaning that
$u-u_0\in W_0^{1,p}(\Omega)$. If $u_0=0$ we retrieve problem
\eqref{P1}.

By a strong generalized solution to problem \eqref{P3} we mean any
$u\in W^{1,p}(\Omega)$ such that  $u-u_0\in W_0^{1,p}(\Omega)$ and
there exits a sequence $\left\{w_n\right\}_{n \in \mathbb{N}}$ in
$W_0^{1, p}(\Omega)$ satisfying the requirements $(\tilde{a})$ and
$(\tilde{c})$ with $w=u-u_0$ and $(\tilde{b})$ with
$T(w_n)=w_n+u_0$. If $u_0=0$, Definition \ref{D1} is recovered.

\begin{theorem}\label{T2}
Assume that condition $(H1)$ is satisfied with $\alpha=\beta=p-1$
and
$$
\max\{2^{p-2},1\}\left(a_1S_{p^*}^{p-1}S_{\frac{p^*}{p^*-p+1}}+a_2S_1\right)<
1.
$$
Then there exists a strong generalized solution $u\in
W^{1,p}(\Omega)$ to problem \eqref{P3} in the sense mentioned above.
\end{theorem}

\begin{proof}
We observe that $p-1<p^*-1$ and
$p-1<\frac{p}{\left(p^*\right)^{\prime}}$, thus $\alpha=\beta=p-1$
is admissible for condition  $(H1)$. The substitution $w=u-u_0$
reduces problem \eqref{P3} to finding $w\in W_0^{1,p}(\Omega)$ that
resolves problem \eqref{P2} with $T(w)=w+u_0$. We note that the map
$T:W_0^{1, p}(\Omega) \rightarrow W^{1, p}(\Omega)$ given by
$T(w)=w+u_0$ for all $w\in W_0^{1,p}(\Omega)$ is continuous and
satisfies hypothesis $(H2)$ with $\alpha=\beta=p-1$ since
$$
\|w+u_0\|_{p^*}^{p-1}\leq \max\{2^{p-2},1\}(S_{p^*}^{p-1}\|\nabla
w\|_p^{p-1}+\|u_0\|_{p^*}^{p-1})
$$
and
$$
\|\nabla(w+u_0)\|_p^{p-1}\leq \max\{2^{p-2},1\}(\|\nabla
w\|_p^{p-1}+\|\nabla u_0\|_p^{p-1}).
$$
\end{proof}
Then Corollary \ref{C1} yields the existence of a sequence
$\left\{w_n\right\}_{n \in \mathbb{N}}$ in $W_0^{1, p}(\Omega)$ and
of a function $w\in W_0^{1, p}(\Omega)$ such that conditions
$(\tilde{a})$-$(\tilde{c})$  hold. The conclusion of Theorem
\ref{T2} is thus achieved.
 \qed

Our second application deals with the following Dirichlet problem
involving competing operators, convection and convolution with
$\rho\in L^1(\mathbb{R}^N)$:
\begin{equation}\label{P4}
\begin{cases}-\Delta_p u + \Delta_q u=f(x,\rho\ast u, \nabla(\rho\ast u)) & \text { in } \Omega \\
u=u_0 & \text { on } \partial \Omega.\end{cases}
\end{equation}
The data $\Omega$, $-\Delta_p$, $-\Delta_q$, and $f(x,s,\xi)$ are as
in the statement of problem \eqref{P1}. In \eqref{P4} we have the
convolution
\begin{eqnarray*}
\rho\ast u(x)=\int_{\mathbb{R}^N}\rho(x-y)u(y)dy \quad \text{for all
$x\in \mathbb{R}^N$},
\end{eqnarray*}
of $\rho\in L^1(\mathbb{R}^N)$ and a $u\in W_0^{1,p}(\Omega)$
regarded as belonging to $W^{1,p}(\mathbb{R}^N)$ by identifying $u$
with its extension by $0$ outside $\Omega$.

The Sobolev-Galiardo-Nirenberg theorem (see \cite[Theorem
9.9]{Brezis}) provides the existsence of a constant $S=S(p,N)$ such
that
\begin{equation}\label{3}
\Vert v\Vert_{L^{p^*}(\mathbb{R}^N)}\leq S\Vert\nabla v
\Vert_{L^p(\mathbb{R}^N)},\quad \forall v\in W^{1, p}(\mathbb{R}^N).
\end{equation}

Our result on problem \eqref{P4} reads as follows.

\begin{theorem}\label{T3}
Assume that condition $(H1)$ is satisfied with $\alpha=\beta=p-1$
and
$$
N^{p-1}\Vert \rho
\Vert_{L^1(\mathbb{R}^N)}^{p-1}\left(a_1S^{p-1}S_{\frac{p^*}{p^*-p+1}}+a_2S_1\right)<
1.
$$
Then there exists a strong generalized solution $u\in
W_0^{1,p}(\Omega)$ to problem \eqref{P3} in the sense of Definition
\ref{D1} corresponding to the operator $T:W_0^{1, p}(\Omega)
\rightarrow W^{1, p}(\Omega)$ given by $Tv=\rho\ast v$ for all $v\in
W_0^{1, p}(\Omega)$.
\end{theorem}

\begin{proof}
As already pointed out in the proof of Theorem \ref{T2}, it is
permitted to take $\alpha=\beta=p-1$ in condition $(H1)$.  Using
Young's theorem for convolution
\begin{equation*}
\Vert \rho \ast v \Vert_{L^p(\mathbb{R}^N)}\leq \Vert \rho
\Vert_{L^1(\mathbb{R}^N)} \Vert v \Vert_{L^p(\mathbb{R}^N)}
\end{equation*}
and the formula
\begin{eqnarray}\label{2}
\frac{\partial}{\partial x_i}(\rho\ast v)=\rho\ast \frac{\partial
v}{\partial x_i}
\end{eqnarray}
for all $v\in W_0^{1, p}(\Omega)$ and $i=1,\dots,N$ (see
\cite[Theorem 4.15]{Brezis}), it turns out that the linear operator
$Tv=\rho\ast v$ is well defined and continuous from $W_0^{1,
p}(\Omega)$ into $W^{1, p}(\Omega)$.

By the convexity of the function $t\mapsto t^p$ on $(0,+\infty)$ and
\eqref{2}, it follows that
\begin{eqnarray*}
&&\Vert\nabla(\rho\ast u)\Vert_{L^p(\mathbb{R}^N)}^p=
\int_{\mathbb{R}^N}\vert \nabla(\rho \ast u) \vert^p
dx=\int_{\mathbb{R}^N}
\left(\sum\limits_{i=1}^N\left(\rho\ast\frac{\partial u}{\partial
x_i}\right)^2\right)^{\frac{p}{2}}dx\\
&&\leq \int_{\mathbb{R}^N}\left(\sum\limits_{i=1}^N\left\vert \rho
\ast \frac{\partial u}{\partial x_i}\right\vert\right)^p dx\leq N^p
\Vert \rho \Vert_{L^1(\mathbb{R}^N)}^p \Vert\nabla u
\Vert_{L^p(\mathbb{R}^N)}^p=N^p \Vert \rho
\Vert_{L^1(\mathbb{R}^N)}^p \Vert\nabla u \Vert_p^p.
\end{eqnarray*}
Then by \eqref{3} we derive
$$
\|\rho\ast u\|_{p^*}^{p-1}\leq \|\rho\ast
u\|_{L^{p*}(\mathbb{R}^N)}^{p-1}\leq S^{p-1}\|\nabla(\rho\ast
u)\|_{L^p(\mathbb{R}^N)}^{p-1} \leq S^{p-1}N^{p-1}\Vert \rho
\Vert_{L^1(\mathbb{R}^N)}^{p-1}\|\nabla u\|_p^{p-1}
$$
and
$$
\|\nabla(\rho\ast u)\|_{L^p(\Omega)}^{p-1}\leq \|\nabla(\rho\ast
u)\|_{L^{p}(\mathbb{R}^N)}^{p-1} \leq N^{p-1}\Vert \rho
\Vert_{L^1(\mathbb{R}^N)}^{p-1}\|\nabla u\|_p^{p-1}
$$
for all $u \in W_0^{1, p}(\Omega)$. Hypothesis $(H_2)$ is thus
verified with $K_1=S^{p-1}N^{p-1}\Vert \rho
\Vert_{L^1(\mathbb{R}^N)}^{p-1}$, $K_2=N^{p-1}\Vert \rho
\Vert_{L^1(\mathbb{R}^N)}^{p-1}$, $K_3=0$. Hence Theorem \ref{T1}
can be invoked with the choice of $T$ described above through which
we reach the stated conclusion.
\end{proof}
\qed

\noindent  \textbf{Acknowledgements:} The authors thank Luiz
Fernando Faria and Anderson Luiz de Albuquerque Araujo for useful
conmments.

\noindent \textbf{Disclosure statement:}
On behalf of all authors, the corresponding author states that there is no conflict of interest.

\noindent  \textbf{Funding:} The authors take part in the project
Special Visiting Researcher – FAPEMIG CEX APQ 04528/22.

\end{document}